\documentclass[reqno, 11pt]{amsart}
\usepackage{mathtools}
\usepackage{amsmath}
\usepackage{amssymb}
\usepackage{yhmath}
\usepackage{graphicx}
\usepackage{ mathrsfs }
\usepackage{bbm}
\usepackage{xcolor}
\usepackage{hyperref}

\DeclareMathAlphabet{\mathpzc}{OT1}{pzc}{m}{it}

\usepackage{thmtools}
\usepackage{thm-restate}

\usepackage{caption}

\newtheorem{theorem}{Theorem}[section]

\newtheorem*{claim*}{Claim}

\newtheorem{lemma}[theorem]{Lemma}
\newtheorem{lem}[theorem]{Lemma}

\newtheorem{Q}[theorem]{Question}

\theoremstyle{definition}
\newtheorem{definition}[theorem]{Definition}\newtheorem{Def}[theorem]{Definition}

\theoremstyle{remark}
\newtheorem{remark}[theorem]{Remark}

\newtheorem{Rmk}[theorem]{Remark}

\numberwithin{equation}{section}

\newcommand{\op}{\operatorname}

\newcommand{\be}{\begin{equation}}
\newcommand{\ee}{\end{equation}}
\newcommand{\Ga}{\Gamma}
\newcommand{\bc}{\mathbb C}
\renewcommand{\c}{\bc}
\newcommand{\R}{\mathbb R}
\renewcommand{\H}{\mathbb H}

\newcommand{\N}{\mathbb N}
\newcommand{\ga}{\gamma}

\newcommand{\La}{\Lambda}
\newcommand{\inte}{\op{int}}
\newcommand{\ba}{\backslash}

\newcommand{\cal}{\mathcal}
\newcommand{\br}{\mathbb R}
\newcommand{\SO}{\op{SO}}
\newcommand{\Isom}{\op{Isom}}

\newcommand{\F}{\cal F}

\newcommand{\bH}{\mathbb H}

\newcommand{\Stab}{\op{Stab}}

\newcommand{\G}{\Gamma}

\newcommand{\BMS}{\op{BMS}}
\renewcommand{\L}{\mathcal L}
\newcommand{\fa}{\mathfrak a}

\renewcommand{\S}{\mathbb S}

\newcommand{\so}{\SO^\circ}

\newcommand{\Gr}{\Gamma_\rho}
\newcommand{\C}{\cal C}

\newcommand{\id}{\op{id}}

\newcommand{\Mob}{\op{M\ddot{o}b}}

\newcommand{\E}{\cal E}

\newcommand{\fg}{\mathfrak g}
\newcommand{\fp}{\mathfrak p}
\newcommand{\fk}{\mathfrak k}
\newcommand{\D}{\Delta}
\newcommand{\yD}{\Upsilon_\D}
\begin{document}

\title[Rigidity of Kleinian groups via self-joinings]{Rigidity of Kleinian groups via self-joinings: measure theoretic criterion}

\author{Dongryul M. Kim}
\address{Department of Mathematics, Yale University, New Haven, CT 06511, USA}
\email{dongryul.kim@yale.edu}

\author{Hee Oh}
\address{Department of Mathematics, Yale University, New Haven, CT 06511, USA}
\email{hee.oh@yale.edu}
\thanks{
 Oh is partially supported by the NSF grant No. DMS-1900101 and 2450703.}
 
\begin{abstract} Let $n, m\ge 2$. Let $\Gamma<\text{SO}^\circ(n+1,1)$ be a Zariski dense convex cocompact subgroup and $\La\subset\S^n$ be its limit set.
Let $\rho : \Gamma \to \text{SO}^\circ(m+1,1)$ be a Zariski dense convex cocompact faithful representation and $f:\Lambda\to \mathbb{S}^{m}$ the $\rho$-boundary map. Let $$\Lambda_f:= \bigcup \left\{ C \cap \Lambda : \begin{matrix}
C \subset \mathbb{S}^n \text{ is a circle such that} \\
f(C \cap \Lambda) \text{ is contained in a proper sphere } \text{in } \mathbb{S}^m
\end{matrix}
\right\}.$$
When there exists at least one
$\La$-doubly stable circle in $\S^n$ (e.g., $\Omega=\mathbb{S}^n-\Lambda$ is disconnected), we prove the following dichotomy: $$\text{either}\quad \Lambda_f= \Lambda \quad \text{ or } \quad \cal H^{\delta}(\Lambda_f) =0,$$ where ${\cal H}^\delta$ is the Hausdorff measure of dimension $\delta=\dim_H \Lambda$. Moreover, in the former case, we have $n=m$ and $\rho$ is a conjugation by a M\"obius transformation on $\mathbb{S}^n$.  Our proof uses ergodic theory for directional diagonal flows and conformal measure theory of discrete subgroups of higher rank semisimple Lie groups, applied to the self-joining subgroup $\Gamma_\rho=(\operatorname{id} \times \rho)(\Gamma) < \text{SO}^\circ(n+1,1)\times  \text{SO}^\circ(m+1,1)$.
We also obtain an analogous theorem for any divergence-type subgroup.
\end{abstract}

\maketitle
\section{Introduction}
Let $\bH^{n+1}$ denote the $(n+1)$-dimensional real hyperbolic space for $n\ge 2$.
The group of its orientation-preserving isometries is given by the identity component
$\so(n+1,1)$ of the special orthogonal group. 
A discrete subgroup $\G<\so(n+1,1)$ is called {\it convex cocompact}
if the convex core\footnote{The convex core of $\Ga\ba \bH^{n+1}$
is the smallest convex submanifold containing all closed geodesics.} 
of the associated hyperbolic manifold $\Ga\ba \bH^{n+1}$ is compact.
Let $\G<\so(n+1,1)$ be a  Zariski dense convex cocompact subgroup for $n\ge 2$, and
$$\rho:\Ga\to \so (m+1,1)$$
be a faithful representation such that $\rho(\Ga) $ is a  Zariski dense convex cocompact cocompact subgroup of
$\so(m+1,1)$ where $m\ge 2$.
For simplicity, we will call a discrete faithful representation $\rho:\Ga\to \so(m+1,1)$  a deformation of $\Ga$ into $\so(m+1,1)$ and a (resp. Zariski dense) convex cocompact deformation of $\Ga$ if the image of $\rho$ is  a (resp. Zariski dense)  convex cocompact subgroup. 
If $\G<\so(n+1,1)$ is cocompact  and $n = m$, the Mostow strong rigidity theorem \cite{Mostowbook} says that $\rho$ is always algebraic, more precisely, it is given by a conjugation by a M\"obius transformation on $\S^n$. However in other cases, Marden's isomorphism theorem and the Teichm\"uller theory imply
that there exists a continuous family of convex cocompact deformations, modulo the conjugations by M\"obius transformation on $\S^m$ (cf. \cite[Section 5]{Matsuzaki1998hyperbolic}).

Let $\La\subset \S^n$ denote
the limit set of $\G$, which is the set of all accumulation points of $\Ga(o)$ in $\S^n$, $o\in \bH^{n+1}$.
 Let $\cal H^{\delta}$ be the $\delta$-dimensional Hausdorff measure on $\S^n$, 
where $\delta$ is the Hausdorff dimension of $\La$ with respect to the spherical metric on $\S^n$. 
Sullivan \cite[Theorem 7]{Sullivan1979density} showed that for $\Ga$ convex cocompact,
we have $$0<\cal H^{\delta}(\La)<\infty .$$

The main aim of this paper is to present a criterion on when $\rho$ is algebraic,
in terms of the Hausdorff measure of the union of all circular slices of $\La$ that are mapped into circles, or more generally into some proper spheres in $\S^m$ by the $\rho$-boundary map.
More precisely, by Tukia \cite{Tukia1985isomorphisms}, there is a unique $\rho$-equivariant continuous embedding  $$f:\La\to \S^m,$$
called the $\rho$-boundary map.
 We consider all circular slices of $\La$ which are
 mapped into some proper spheres in $\S^m$ by $f$:
$$
\Lambda_f:= \bigcup \left\{ C \cap \La : \begin{matrix}
C \subset \S^n \mbox{ is a circle such that} \\
f(C \cap \La) \mbox{ is contained in a proper sphere } \text{in $\S^m$}
\end{matrix}
\right\}.
$$
\vspace{-1.5em}
 \begin{figure}[h]
  \includegraphics [height=3.5cm]{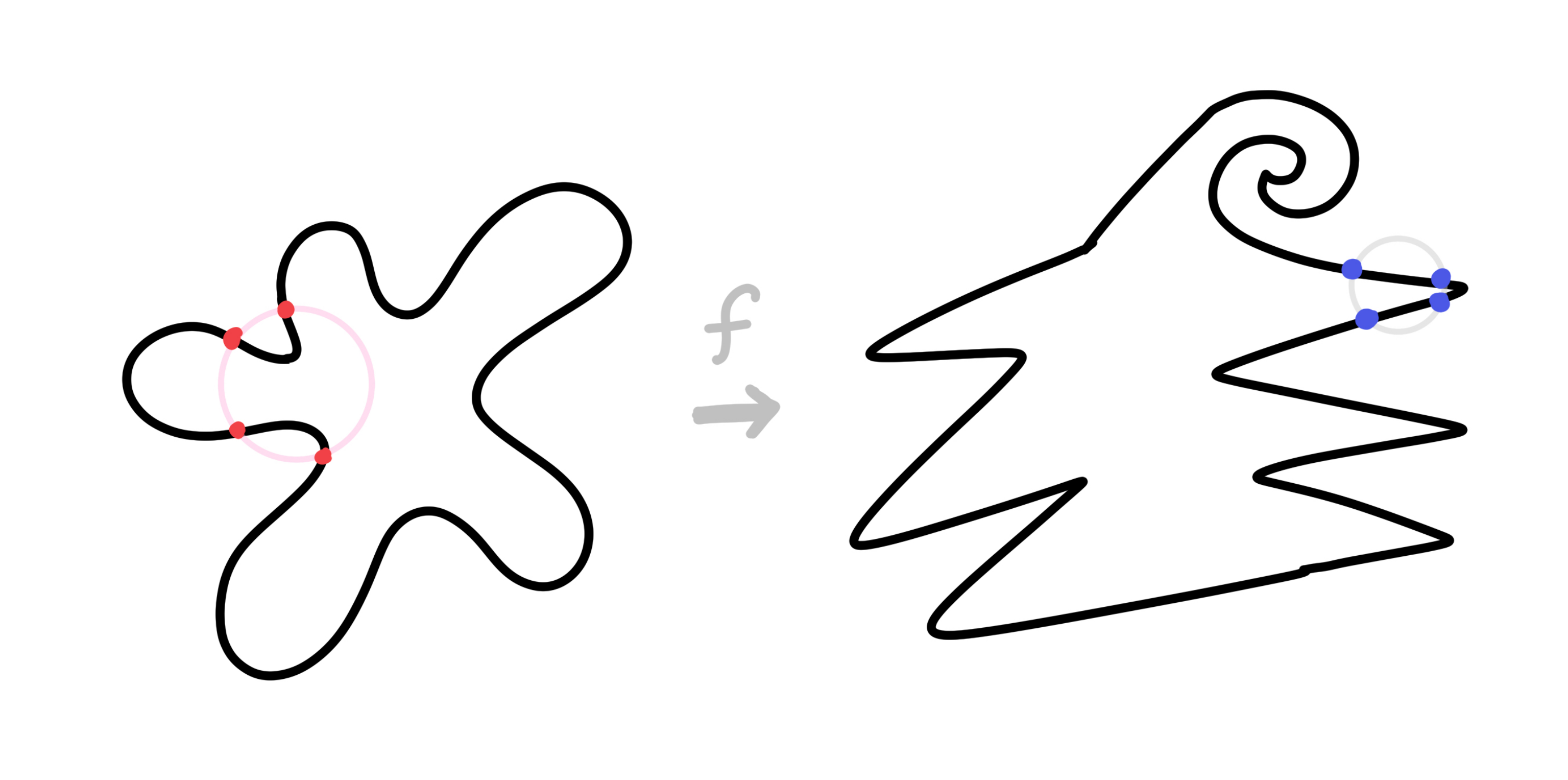}
\caption{$f(C\cap \La)$ is contained in a circle} 
\label{fig.slice}
\end{figure}

We emphasize that the boundary map $f$ is defined only on $\La$ and therefore
our definition of $\La_f$ involves the image of the intersection $C\cap \Lambda$ under $f$, but not the whole circle $C$ (see Figure \ref{fig.slice}).
If $n=m$ and $f$ is a M\"obius transformation of $\S^n$,  then
 $f$ clearly maps all circles to circles and hence $\La_f=\La$.
 The following main theorem of this paper says that in all other cases,
$\La_f$ has  zero $\cal H^{\delta}$-measure. In other words, if $\cal H^{\delta}(\La_f)>0$,
then $f$ is the restriction of a M\"obius transformation of $\S^n$ and $\rho$ is algebraic.
\begin{theorem} \label{high} Let $n, m\ge 2$. Let $\G<\so(n+1,1)$ be a  Zariski dense convex cocompact subgroup
such that the ordinary set $\Omega=\S^n-\Lambda$ has at least two components.
Let $\rho:\Ga\to \so (m+1,1)$
be a Zariski dense convex cocompact deformation and $f:\La\to \S^m$ the $\rho$-boundary map.
 Then
$$\text{either}\quad  \La_f=
\La \quad \text{ or } \quad \cal H^{\delta}(\La_f) =0.$$

In the former case, we have
 $n=m$, $f$ extends to some $g\in \Mob(\S^n)$ and $\rho$ is a conjugation by $g$.
 \end{theorem}

When $n=m=2$, the topological version of the above theorem that either $\La_f = \La$ or $\La_f$ has empty interior was obtained in our earlier paper \cite{kim2022rigidity} for all finitely generated discrete subgroups. Theorem \ref{high} provides its measure theoretic version. See Theorem \ref{thm.topological} for the topological version for general $n, m\ge 2$.
 
\begin{Rmk}
 If $\Gamma<\so(3,1)$ is convex cocompact with $\La$ connected, then $\Omega$ is disconnected \cite[Chapter IX]{Maskit1988Kleinian}; hence Theorem \ref{high} applies. \end{Rmk}

Indeed, we prove Theorem \ref{high} under a weaker condition that
 there exists a $\La$-doubly stable circle (Theorem \ref{high2}). 
 \begin{definition}
We say that a circle $C\subset \S^n$ is $\La$-{\it doubly stable} if 
for any sequence of circles $C_k$ converging to $C$, 
$$\#\limsup (C_k\cap \La) \ge 2,$$
where $\limsup E_k$ is defined as $\bigcap_{n \in \N} \overline{\bigcup_{k \ge n} E_k} $ for a sequence $E_k\subset \S^n$.
\end{definition}
If $\Omega$ is disconnected,
there exists a $\La$-doubly stable circle (Lemma \ref{lem.doubly}). If $\Omega=\emptyset$, i.e., $\La=\S^n$,
then every circle is $\S^n$-doubly stable.
In particular, Theorem \ref{high2} applies to  any uniform lattice $\Ga$ of $\so(n+1, 1)$:  either $f : \S^n \to \S^m$ preserves Lebesgue-almost none of the circles, or $n = m$ and $f$ is induced by a M\"obius transformation on $\S^n$.
\begin{remark}
It is an interesting question whether there exists a Zariski dense convex cocompact subgroup of $\SO^\circ (3,1)$ whose limit set $\La$ is totally disconnected and  there is no $\La$-doubly stable circle.
\end{remark}

In terms of the quasiconformal deformation indicated in Figure \ref{fig.QC},
our theorem implies that the union of circular slices of the left limit set which are mapped into circles has zero $\cal H^\delta$-measure.

 \begin{figure}[h]
  \includegraphics [height=4cm]{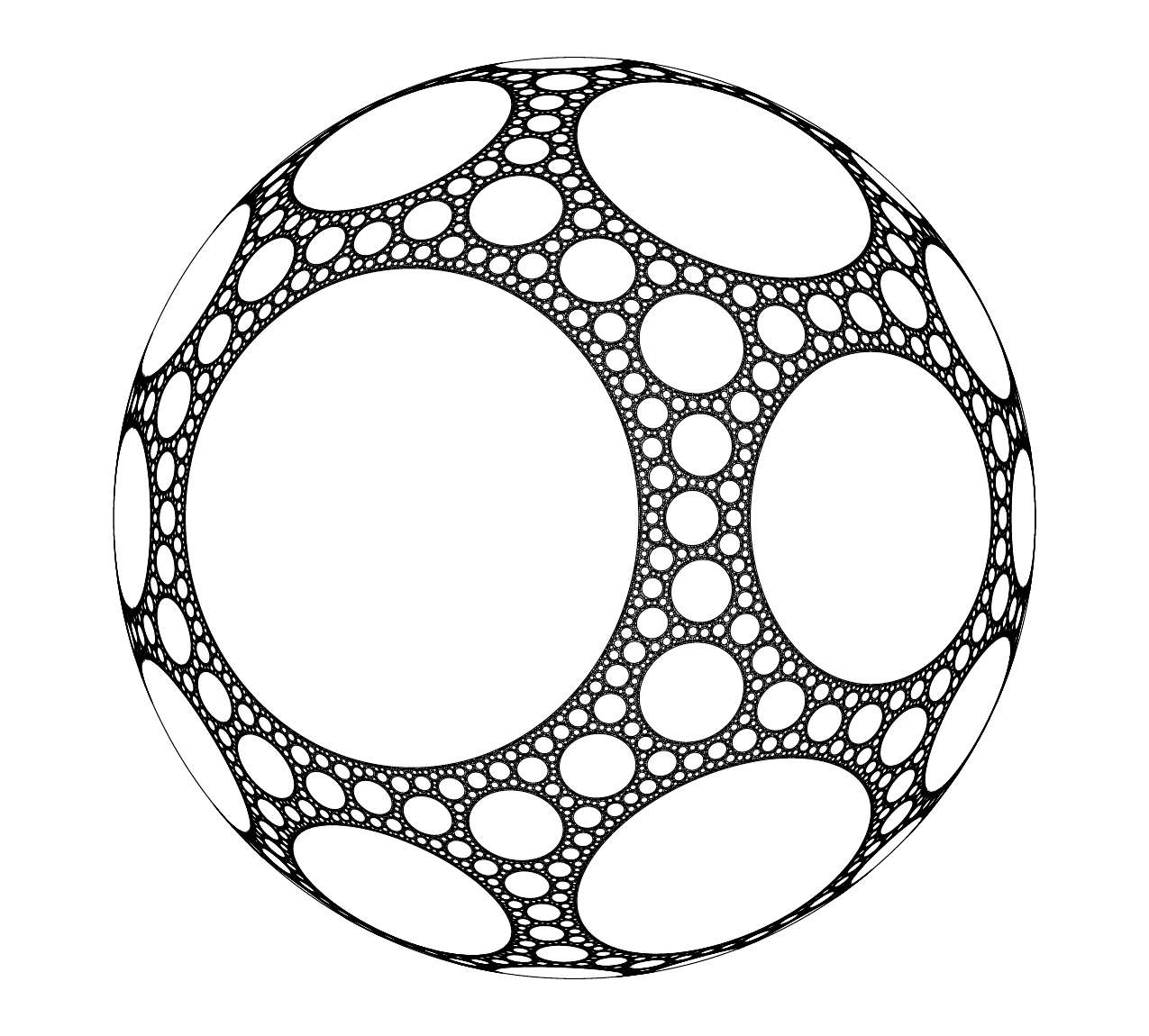}
  \includegraphics [height=4cm]  {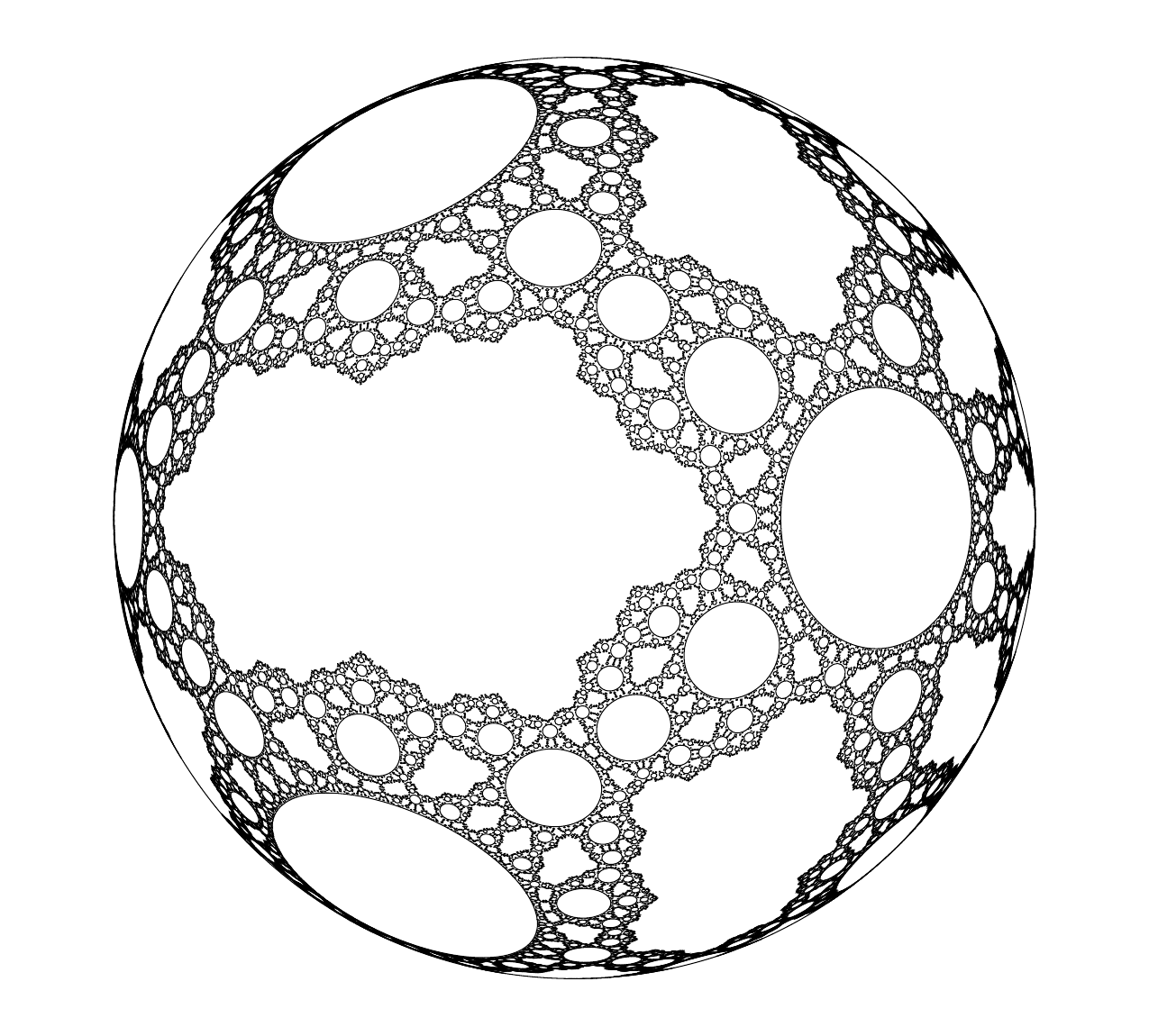} 
  
  \caption[Non-trivial quasiconformal deformation]{Non-trivial quasiconformal deformation\footnotemark} \label{fig.QC}
\end{figure}
 \footnotetext{Image credit: Curtis McMullen and Yongquan Zhang \cite{Zhang2022construction}}

 Note that $(n+2)$-distinct points on $\S^n$ form the set of vertices of a unique
 ideal hyperbolic $(n+1)$-simplex of $\H^{n+1}$. Gromov-Thurston's proof of Mostow rigidity theorem (\cite{Gromov1981hyperbolic}, \cite{thurston2022geometry}) uses the fact that a homeomorphism of $\S^n$ mapping vertices of every maximal volume $(n+1)$-simplex of $\H^{n+1}$ to vertices of a maximal volume $(n+1)$-simplex is a M\"obius transformation. 
 
Any $(n+2)$-distinct points on $\S^n$ form vertices of a zero-volume $(n+1)$-simplex of $\bH^{n+1}$ if and only if they lie in some codimension one sphere in $\S^n$. We also prove the following higher dimensional version of \cite[Theorem 1.3]{kim2022rigidity}, which answered McMullen's question for $n=2$:

\begin{theorem} \label{thm.McMullen} Let $\G<\so(n+1,1)$ be a  Zariski dense discrete subgroup. Suppose that there exists a $\La$-doubly stable circle in $\S^n$.
If the $\rho$-boundary map $f : \La \to \S^n$ maps vertices of every $(n+1)$-simplex of zero-volume to vertices of an $(n+1)$-simplex of zero-volume, then $f$ extends to a M\"obius transformation of $\S^n$.
\end{theorem}

We obtain a stronger statement  that unless  $f$ extends to a M\"obius transformation, the union of all vertices of $(n+1)$-simplexes of zero-volume whose images under $f$ form vertices of zero-volume $(n+1)$-simplexes has empty interior in $\Lambda$.

\subsection*{On the proof of Theorem \ref{high}}
We use the theory of Anosov representations.
Consider the following self-joining subgroup of $G = \so(n + 1, 1) \times \so(m + 1, 1)$:
$$\Gr := (\id \times \rho)(\Ga) = \{(\ga, \rho(\ga)) : \ga \in \Ga\}.$$
The crucial point is that, under the assumption that both $\Ga$ and $\rho(\Ga)$ are Zariski dense and convex cocompact and not conjugate to each other, we have that
\begin{multline*}
\text{$\Gr$ is  a Zariski dense
{\it Anosov} subgroup of $G$} \\
\text{with respect to a minimal parabolic subgroup }\end{multline*}(see the discussion around \eqref{ano}). 
Hence the recent classification theorem on higher rank conformal measures by Lee-Oh \cite{lee2020invariant} (Theorem \ref{bi}) and the ergodicity theorem of Burger-Landesberg-Lee-Oh \cite{burger2021hopf} (Theorem \ref{er}) apply to our setting, yielding that for any $\Gr$-conformal measure
on the limit set of $\Gr$, the associated Bowen-Margulis-Sullivan measure on $\Ga_\rho\ba G$ is ergodic  and conservative
for a unique one-parameter  diagonal flow $A_u=\{\exp tu:t\in \br\}$ where $u$ is a vector in the interior of the positive Weyl chamber.

A general higher rank conformal measure seems mysterious. However, the graph structure of our self-joining group $\Gr$ allows us to pin down a very explicit $\Ga_\rho$-conformal measure, which we call the graph-conformal measure \cite{kim2023rigidity}. 
Indeed, under the convex cocompactness hypothesis on $\Ga$,
the graph-conformal measure
is given by the pushforward measure $(\id \times f)_* (\cal H^{\delta}|_{\La})$, and this is the reason why we can relate
the Hausdorff measure $\cal H^\delta|_\La$ with dynamics on the
Anosov homogeneous space $\Ga_\rho\ba G$ in the proof of Theorem \ref{high}.

The conclusion of Theorem \ref{high} follows if
we show that $\Ga_\rho$ cannot be Zariski dense in $G$ (Lemma \ref{Zdense}).
We give a proof by contradiction. Suppose that $\Ga_\rho$ is Zariski dense.
Considering the action of $\Ga_\rho$ on
the space $\Upsilon_{\rho}$ of all ordered pairs $Y = (C, S)$ of a circle $C \subset \S^n$ and a codimension one sphere $S \subset \S^m$ intersecting the limit set $\La_{\rho} \subset \S^n \times \S^m$ of $\Gr$,
we are then able to prove, together with the work of Guivarch-Raugi \cite{Guivarch2007actions} and the aforementioned ergodicity  and conservativity result for the directional diagonal flows, that
for $\cal H^\delta|_\La$-almost all $\xi\in \La$, the $\Ga_\rho$-orbit of $Y\in \Upsilon_\rho$
containing $(\xi, f(\xi))$ is dense in the space $\Upsilon_\rho$.

On the other hand, we show that the existence of a $\La$-doubly stable circle in $\S^n$
implies that
     for any $Y_0 = (C_0, S_0) \in \Upsilon_{\rho}$ with $f(C_0 \cap \La) \subset S_0$,
    the orbit  $\Gr Y_0$ cannot be dense in $\Upsilon_{\rho}$ (Theorem \ref{thm.wKM}).
      This shows that $\Ga_\rho$ cannot be Zariski dense when $\cal H^\delta
      (\La_f)>0$. We also show that
      when $\Omega$ is disconnected, a $\La$-doubly stable circle exists
      (Lemma \ref{lem.doubly}).

\subsection*{Analogous question for rational maps} 
We close the introduction by the following question which seems natural in view of Sullivan's dictionary between Kleinian groups and rational maps (\cite{Sullivan1985quasiconformal}, \cite{McMullen1994classification}).
\begin{Q}\rm   Let $h_1, h_2:\hat \c \to \hat \c $ be 
rational  maps of degree at least $2$ whose  Julia sets 
are not contained in circles.
Suppose that  $h_2= F\circ h_1 \circ F^{-1} $
for some quasiconformal homeomorphism $F:\hat \c\to \hat \c$.
Suppose that  for the Julia set $J=J_{h_1}$ of $h_1$, there exists a $J$-doubly stable circle in $\hat \c$.
Let $$J_F:=\bigcup \left\{ C \cap J : \begin{matrix}
C \subset \hat \c \mbox{ is a circle such that} \\
F(C \cap J) \mbox{ is contained in a circle}
\end{matrix}
\right\} .$$

\begin{enumerate}
    \item If $J_F=J$, is  $F\in \Mob (\hat \c)$?

\item Suppose that $h_1, h_2$ are hyperbolic. Let $\delta=\text{dim}_H J$.
Is it true that
$$\text{either}\quad  J_F=
J \quad \text{ or } \quad \cal H^{\delta}(J_F) =0 ?$$
\end{enumerate}

\end{Q}

\subsection*{Added in proofs:}
   Using recent developments in the ergodic theory of transverse subgroups (\cite{CZZ_transverse}, \cite{KOW_indicators}, \cite{kim2024conformal}), Theorem \ref{high} can be extended to all discrete subgroups $\Ga$ of divergence type and  quasi-isometric deformations $\rho$, provided $\cal H^\delta$ is replaced by the unique $\delta$-dimensional $\Ga$-conformal measure on $\La$, where $\delta$ is the critical exponent of $\Ga$ \cite{Sullivan1979density}.
This covers all geometrically finite groups and type-preserving deformations. See Theorem \ref{high5}.

\subsection*{Organization}
The main goal of section \ref{sec.erg} is to prove Theorem \ref{ttt}, which we deduce 
from the classification of conformal measures in \cite{lee2020invariant} and  the ergodicity and conservativity of directional diagonal flows in \cite{burger2021hopf} with respect to the Bowen-Margulis-Sullivan measure associated to the $\Ga_\rho$-conformal measure constructed from the $\delta$-dimensional Hausdorff measure on $\La$. The main theorem of
 section \ref{sec.orbit} is Theorem \ref{ett} which we deduce from Theorem \ref{ttt} and a theorem of Guivarch-Raugi (Theorem \ref{GR}).
In section \ref{sec.wKM}, we discuss an obstruction to dense $\Ga_\rho$-orbits in the space $\Upsilon_\rho$
when a $\La$-doubly stable circle exists.
 In section \ref{sec.proof}, we give a proof of Theorem \ref{high}. We also discuss  a topological version of Theorem \ref{high} without convex cocompactness assumption (Theorem \ref{thm.topological}).

\subsection*{Acknowledgement} We would like to thank Curt McMullen for useful comments on the preliminary version. We are also grateful to him  and Yongquan Zhang for allowing us to use the beautiful image of Figure \ref{fig.QC}. 

\section{Ergodicity and graph-conformal measure} \label{sec.erg}

Let $(X_1,d_1)$ and $(X_2,d_2)$ be rank one Riemannian symmetric spaces.
Let $G$ be the product $G_1\times G_2$
where $G_1=\Isom^\circ (X_1)$ and $G_2=\Isom^\circ (X_2)$ are connected simple real algebraic groups of rank one.
Then $G=\Isom^\circ X$ where $X=X_1\times X_2$ is the Riemannian product.
 We fix a Cartan involution $\theta$ of the Lie algebra $\mathfrak{g}$ of $G$, and decompose $\fg$ as $\mathfrak g=\mathfrak k\oplus\mathfrak{p}$, where $\fk$ and $\fp$ are the $+ 1$ and $-1$ eigenspaces of $\theta$, respectively. 
We denote by $K$ the maximal compact subgroup of $G$ and choose a maximal abelian subalgebra $\fa$ of $\mathfrak p$.
Choosing a closed positive Weyl chamber $\fa^+$ of $\fa$, let $A:=\exp \mathfrak a$ and $A^+=\exp \mathfrak a^+$. The centralizer of $A$ in $K$ is denoted by $M$, and we let
$N^+$ and $ N =  N^-$ be the horospherical subgroups so that $\log   N^+ $ and
$\log N^-$ are  the sum of all negative and positive root subspaces for our choice of $A^+$ respectively.
We set $$P^+=MAN^+,\quad\text{and} \quad P=P^-=MAN;$$ they are minimal parabolic subgroups of $G$ that are opposite to each other.
The quotient $\F=G/P$ is known as the Furstenberg boundary of $G$, and is isomorphic to $K/M$. Let $\op N_K(\fa)$  be the normalizer of $\fa$ in $K$ and let $\cal W:=\op N_K(\fa)/M$ denote the Weyl group.
Let $w_0\in \op N_K(\fa)$ be the unique element in $\cal W$ such that
$w_0Pw_0^{-1}=P^+$.
  For each $g\in G$, we define 
   $$g^+:=gP\in \F \quad\text{and}\quad g^-:=gw_0P\in \F.$$

An element $g\in G$ is
loxodromic if $g=h a m h^{-1}$ for some $ a\in \inte A^+$,
$m\in M$ and $h\in G$. The Jordan projection of $g$
is defined to be $\lambda(g):=\log a \in \inte \fa^+$. 

In the rest of the section, let $\Delta$ be  a Zariski dense
discrete subgroup of $G$. 
 The {\it{limit cone}} $\L_{\D}\subset\fa^+$  is defined as the smallest closed cone containing all Jordan projections of loxodromic elements of $\D$. It is a convex subset of $\fa^+$ with non-empty interior \cite[Section 1.2]{Benoist1997proprietes}. Benoist showed that there exists a unique $\D$-minimal subset of $\F$, which is called the limit set of $\Delta$. 
 We denote it by $\La_\Delta$. 

\subsection*{Bowen-Margulis-Sullivan measures} Let $\F_i$ be the Furstenberg boundary of $G_i$, which is equal to the geometric boundary $\partial X_i$.
For each $i=1,2$, the Busemann function $\beta_{\xi_i}(x_i, y_i)$ is defined
as
\be \label{eqn.rankonebuse}
\beta_{\xi_i}(x_i, y_i)=\lim_{t\to \infty} d_i(\xi_{i,t}, x_i) -d_i (\xi_{i, t}, y_i)
\ee
where $\xi_{i, t}$ is a geodesic ray toward to $\xi_i$.
 For $\xi = (\xi_1, \xi_2) \in \F = \F_1 \times \F_2$ and $x = (x_1, x_2), y = (y_1, y_2) \in X$, the $\fa$-valued Busemann function is defined componentwise: $$\beta_{\xi}(x, y) = (\beta_{\xi_1}(x_1, y_1), \beta_{\xi_2}(x_2, y_2)) \in \fa$$ 
where we have identified $\fa = \fa_1 \oplus \fa_2$ with $\br^2$.

In the following we fix $o=(o_1, o_2)\in X$ so that the stabilizer of $o$ is $K$.
\begin{Def} 
    For a linear form $\psi \in \fa^*$, a Borel probability measure $\nu$ on $\F$ is called a \emph{$(\D, \psi)$-conformal measure} (with respect to $o$) if for any $g \in \D$ and $\xi \in \F$, $${d g_* \nu \over d \nu}(\xi) = e^{\psi(\beta_{\xi}(o, g o))}$$ where $g_*\nu(B) = \nu(g^{-1}B)$ for any Borel subset $B \subset \F$.  By a $\Delta$-conformal measure, we mean a $(\Delta, \psi)$-conformal measure for some $\psi\in \fa^*$.
\end{Def}

Two points $\xi=(\xi_1, \xi_2)$ and
$\eta=(\eta_1, \eta_2)$ are in general position
if $\xi_i\ne \eta_i$ for each $i=1,2$.
Let $\F^{(2)}$ be the set of all pairs $(\xi, \eta)\in \F\times \F$ which are in general position. The map $G\to \F^{(2)} \times \fa$, $g \mapsto (g^+, g^-,  \beta_{g^+}(o, g o))$ induces a $G$-equivariant homeomorphism $G/M\simeq \F^{(2)} \times \fa$, called the Hopf-parametrization.

For a $(\Delta,\psi)$-conformal measure $\nu$ supported on the limit set $\La_\Delta$
for some $\psi \in \fa^*$, we can define the following Borel measure on $G/M$ using the Hopf-parametrization: \be \label{eqn.bms}
d\tilde m^{\BMS}_{\nu}(gM) = e^{\psi(\beta_{g^+}(o, g o) + \beta_{g^-}(o, g o))} d\nu(g^+) d\nu(g^-) db
\ee
where $db$ is the Haar measure on $\fa$.
By integrating over the fiber of $G \to G/M$ with respect to the Haar measure of $M$, we will consider $\tilde m^{\BMS}_{\nu}$ as a Radon measure on $G$, which is then a left $\Delta$-invariant and right $AM$-invariant measure. We denote by $m^{\BMS}_{\nu}$ the Radon measure on $\Delta\ba G$ induced by $\tilde m^{\BMS}_{\nu}$. This measure is called the Bowen-Margulis-Sullivan measure associated to $\nu$. Its support is $$\Omega_{\D}=\{[g]\in \Delta\ba G: g^{\pm}\in \La_\Delta\}.$$ We refer to \cite{edwards2020anosov} for a detailed discussion on the construction of this measure.

\subsection*{Self-joinings of convex cocompact groups}
In the rest of the section, we will consider the following special type of discrete subgroups of $G$. Let $\Ga<G_1$ be a Zariski dense convex cocompact subgroup and
$\rho:\Ga\to G_2$ be a Zariski dense convex cocompact faithful representation.
Define the self-joining of $\Ga$ via $\rho$: $$\Gr := (\id \times \rho)(\Ga) = \{ (\ga, \rho(\ga)) : \ga \in \Ga\}$$ which is a discrete subgroup of $G$. 

It follows from the convex cocompactness assumption for $\Ga$ and $\rho(\Ga)$ that
if we fix a word metric $|\cdot|$ on $\Ga$ for some finite generating set and fix $o_1\in X_1$ and $o_2\in X_2$, then there exist constants $C,C'>0$ such that
 for all $\ga\in \Ga$, 
\be\label{ano}  \min \{d_1 (\ga o_1, o_1), d_2(\rho(\ga) o_2, o_2)\} \ge C |\ga| -C' .\ee
In other words, $\Ga_\rho$
is an Anosov subgroup of $G$ with respect to a minimal parabolic subgroup (\cite{Labourie2006anosov}, \cite{Guichard2012anosov}, \cite{Kapovich2017anosov}). 
This enables us to use the general theory developed for Anosov subgroups. We remark that ergodic theory for self-joining groups of convex cocompact groups was first studied in
\cite{Burger1993intersection}.

Since both $G_1$ and $G_2$ are simple, we have the following equivalence between Zariski density of the self-joining and the rigidity of $\rho$, first observed by Dal'Bo-Kim \cite{DalBoKim_criterion}:

\begin{lem} \label{Zdense}
The subgroup $\Ga_\rho$ is Zariski dense in $G$ if and only if $\rho$ does not extend to a Lie group isomorphism $G_1 \to G_2$. 
\end{lem}

Since $\Ga$ and $\rho(\Ga)$ are convex cocompact, there exists a unique $\rho$-equivariant  continuous embedding $f : \La \to \F_2$; this is a special case of a theorem of Tukia \cite{Tukia1985isomorphisms}, which can also be seen directly as follows. Since $\Ga < G_1$ is convex cocompact, $\Ga$ is a hyperbolic group and an orbit map $\Ga \to X_1$ is a quasi-isometric embedding, where $\Ga$ is equipped with a word metric. Hence, it follows from a standard result for Gromov hyperbolic spaces (e.g. \cite[Chapter III.H, Theorem 3.9]{Bridson1999metric}) that there exists a unique $\Ga$-equivariant homeomorphism $f_1: \partial \Ga \to \La_{\Ga} = \La$ where $\partial \Ga$ is the Gromov boundary of $\Ga$. Similarly, we obtain a unique $\rho(\Ga)$-equivariant homeomorphism $f_2 : \partial \rho(\Ga) \to \La_{\rho(\Ga)}$. Since $\rho : \Ga \to \rho(\Ga)$ is an isomorphism, there exists a unique $\rho$-equivariant homeomorphism $f_0 : \partial \Ga \to \partial \rho(\Ga)$. Therefore, $f := f_2 \circ f_0 \circ f_1^{-1} : \La \to \La_{\rho(\Ga)}$ is the unique $\rho$-equivariant homeomorphism into $\F_2$.

Hence, for $\Ga_\rho$ Zariski dense, its limit set $\La_{\rho} \subset \F$ is of the form $$\La_{\rho} = (\id \times f)(\La)$$
where $\id\times f:\La \to \La_\rho$ is the diagonal embedding.
We denote by $\L_\rho\subset \fa^+$ the limit cone of $\Gr$:
$$\L_\rho=\L_{\Gr}.$$

Since $\Ga_\rho$ is  Anosov,  the following Theorems \ref{bi} and \ref{er} are special cases of theorems proved in those respective papers. Let $\fa^*$ denote the set of all $\R$-linear forms on $\fa$.

\begin{theorem}[Classification of conformal measures, {\cite[Theorem 1.3, Proposition 4.4]{lee2020invariant}}] \label{bi} Suppose that $\Gr$ is Zariski dense in $G$. 
The space of unit vectors in $\inte \L_\rho$ is in bijection with the space of all $\Gr$-conformal measures  on $\La_\rho$. Moreover,  each $\Gr$-conformal measure on
$\La_\rho$ is a $(\Gr, \psi)$-conformal measure for a unique linear form $\psi\in \fa^*$.
\end{theorem}
We will denote this bijection by
\be\label{bij} u\mapsto \nu_u .\ee 
For each unit vector $u\in \inte\L_\rho$,
we also denote by $\psi_u\in \fa^*$ the (unique) linear form associated to $\nu_u$, that is,
$\nu_u$ is $(\Gr, \psi_u)$-conformal.

\subsection*{Ergodicity} For simplicity, we set
$$\tilde m^{\BMS}_{u}:= \tilde m^{\BMS}_{\nu_u}\quad \text{ and }\quad 
 m^{\BMS}_{u}:= m^{\BMS}_{\nu_u}$$
For any non-zero vector $u\in \fa$, we consider the following one-parameter semigroup/subgroup:
$$A_u^+:=\{ a_{tu} : t\ge 0\}\text{ and }A_u:=\{ a_{tu} : t\in \br \} .$$
where $a_{tu}=\exp tu$.
The following ergodicity result due to Burger-Landesberg-Lee-Oh \cite{burger2021hopf} is the  main ingredient of our proof of Theorem \ref{high}:
\begin{theorem}[Ergodicity of directional flows, {\cite{burger2021hopf}}] \label{er} Suppose that $\Gr$ is Zariski dense in $G$. 
For any unit vector $u\in \inte \L_\rho$,
  $( m^{\BMS}_{u}, \Gr\ba G)$ is ergodic  and conservative for the $A_u$-action. In particular,
  for $ m^{\BMS}_{u}$-almost all $x$, $xA_u^+$ is dense in $\Omega_{\Gr}$.
\end{theorem}

\subsection*{Graph-conformal measure} 
Let $\nu_\Ga$ be the $\Ga$-conformal measure supported on the limit set $\La$ of $\Ga$; since $\Ga$ is convex cocompact, it exists uniquely \cite{Sullivan1979density}. It turns out that the  measure 
$ (\id \times f)_*\nu_\Ga $ is a $\Ga_\rho$-conformal measure,
where $\id\times f:\La \to \La_\rho$ is the diagonal embedding.
We called this measure the graph-conformal measure in \cite{kim2023rigidity}. More precisely, we have the following lemma, thanks to which
we were able to apply
Theorem \ref{er} in the proof of Theorem \ref{high}:  we denote by $\delta_{\Ga}$ the critical exponent of $\Ga$.
\begin{lemma}\cite[Proposition 4.9]{kim2023rigidity}\label{lem.pushforward}
The measure $$(\id \times f)_*\nu_\Ga $$ is a 
$(\Gr,\sigma_1)$-conformal measure supported on $\La_{\rho}$, where
 $\sigma_1 \in \fa^*$ is the linear form given by
$\sigma_1(t_1, t_2) = \delta_{\Ga} t_1$ for $(t_1, t_2) \in \fa$.
\end{lemma}

We now deduce Theorem \ref{ttt} from
Theorems \ref{bi} and \ref{er}: first, there exists a unique unit vector \be\label{urho} u_\rho\in \inte \L_\rho
\text{ such that }(\id \times f)_*\nu_\Ga =\nu_{u_\rho} .\ee 
Hence if we write $\Omega_\rho:=\Omega_{\Ga_\rho}=\{[g]\in \Ga_\rho\ba G: g^{\pm}\in \La_\rho\}$, we get the following main theorem of this section:
\begin{theorem}\label{ttt} 
Suppose that $\Gr$ is Zariski dense. Then there exists an $(\id \times f)_*\nu_\Ga$-conull subset
$$\La_{\rho}'\subset \La_\rho$$ such that
for any $g\in G$ with $g^+\in \La_{\rho}'$, the closure $\overline{[g]A_{u_{\rho}}^+}$ contains
$\Omega_\rho$.
\end{theorem}

\begin{proof} Since $\tilde m^{\BMS}_{u_\rho}$ is equivalent to the product measure $d\nu_{u_\rho}\times d\nu_{u_\rho}\times da\times  dm  $ where $da$ and $ dm$ denote Haar measures on $A$ and $M$ respectively, it follows from
 Theorem \ref{er} that there exists a
$\nu_{u_\rho}$-conull subset $\La_\rho'\subset \La_\rho$ such that 
for all $\xi \in \La_{\rho}'$, there exists $g_0 \in G$ with $g_0^+ = \xi$ and $g_0^{-} \in \La_{\rho}$ such that $[g_0]A_{u_{\rho}}^+$ is dense in $\Omega_{\rho}$.
Hence the claim follows by the following Lemma \ref{den}.
\end{proof}

\begin{lem}\label{den} Let $u\in \inte \fa^+$ and $\D<G$ be a Zariski dense discrete subgroup.
If $[g_0]A_u^+$ is dense in $\Omega_\D$, then for any $g\in G$ with $g^+=g_0^+$, the closure $\overline{[g]A_u^+}$ contains $ \Omega_\D$. 
\end{lem}

\begin{proof} This can be deduced from the proof of \cite[Corollary 2.3]{kim2022rigidity}, which we recall for readers' convenience.
    Since $g^+ = g_0^+$, $g = g_0 p$ for some $p \in P$. Writing $p = nam \in NAM$, we claim that
    $$
    (\Omega_{\Delta} - [g_0] A_u^+) ma \subset \overline{[g]A_u^+}.
    $$
    Let $x\in \Omega_{\Delta}-[g_0] A_u^+$. Since $\overline{[g_0]A_u^+} \supset \Omega_{\Delta}$, there exists a sequence $t_i\to +\infty$ such that
$x=\lim_{i\to \infty} [g_0] a_{t_i u} $.
Since $u \in \inte \fa^+$,
we have 
$a_{-t_i u} n a_{t_i u}\to e$ as $i\to \infty$.
Therefore
$$\lim_{i\to \infty} [g] a_{t_i u}= \lim_{i\to \infty} [g_0] nam a_{t_iu}=
\lim_{i\to \infty} [g_0] a_{t_iu} (a_{-t_i u} n a_{t_i u})  am
=xam ;$$
so $xam\in \overline{[g]A_u^+} $.
This proves the claim.

Since
$\Omega_{\Delta}$ is $AM$-invariant and $\Omega_{\Delta}-[g_0]AM$ is dense in $\Omega_{\Delta}$ (as $\La_{\Delta} \subset \F$ is a perfect subset),
it follows that 
$$ \Omega_{\Delta} \subset \overline{[g]A_u^+}.$$
\end{proof}

\section{Orbits in the space of circle-sphere pairs} \label{sec.orbit}

Let  $G_1=\so(n+1,1)$, $n\ge 2$ and $G_2=\so(m+1,1)$, $m\ge 2$.
We set
$$\Upsilon=\{ Y=(C, S): \text{$C\subset \S^n$ a circle, $S\subset \S^m$ a codimension one
sphere}\}.$$
Let $G=G_1\times G_2$.
The group $G$ acts on $\Upsilon$ componentwise: 
$$(g_1, g_2)(C, S)= (g_1C, g_2 S)$$
for $(g_1, g_2)\in G_1\times G_2$ and $(C, S)\in \Upsilon$.
Let $\Delta < G$ be a Zariski dense
discrete subgroup. Then $\Delta$ acts on the space
$$
 \Upsilon_{\Delta} = \{Y\in \Upsilon : Y \cap \La_{\Delta} \neq \emptyset\}, 
 $$
which is a closed subset of $\Upsilon$.

\subsection*{Denseness of $\yD^*$}
Let $$\yD^*:=\{Y\in \yD: \# Y\cap \La_\D \ge 2\}.$$
\begin{theorem}\label{two}
The subset
$\yD^*$ is dense in  $\yD $.
\end{theorem}

Recalling that $P=MAN$ and $\F=G/P\simeq K/M$,
we have $G/AN\simeq K$.
Consider the projection $\pi:G/AN=K\to G/P=K/M$, and set
$$\tilde \La_{\Delta}=\pi^{-1} (\La_{\Delta})\subset G/AN=K.$$
Since $M\simeq \SO(n)\times \SO(m)$ is connected,
the following is a special case of a theorem of Guivarch and Raugi
\cite{Guivarch2007actions}:
\begin{theorem}[{\cite[Theorem 2]{Guivarch2007actions}}] \label{GR} The action of $\Delta$ on $\tilde \La_\Delta$ is minimal.
\end{theorem}
Indeed, this theorem is a key ingredient of the proof of Theorem \ref{two}, which we now begin.

\medskip

\noindent{\bf Proof of Theorem \ref{two}.}
For simplicity, we write $\La$ for $\La_\Delta$ in this proof.
Write $K=K_1\times K_2$ where $K_1=K\cap (G_1\times \{e\})=\SO(n+1)$
and $K_2=K\cap (\{e\}\times G_2)=\SO(m+1)$, and similarly, we write $M=M_1\times M_2=\SO(n)\times \SO(m)$.
Via the projection $K_i\to K_i/M_i=\F_i$, we can think of a point of $K_i$
as an orthonormal frame $\mathsf f_\xi$ based at $\xi\in \F_i$. Hence an element of $K$ is a pair of orthonormal frames $(\mathsf f_{\xi_1}, \mathsf f_{\xi_2})\in K_1\times K_2$.
For an infinite sequence $(\xi_{1, j}, \xi_{2, j}) \in \F_1\times \F_2$ converging to $(\xi_1, \xi_2)$, we say that the convergence is
  $(1, 1)$-tangential to the frame $(\mathsf f_{\xi_1}, \mathsf f_{\xi_2})$  if, for each $i = 1, 2$,
the sequence of unit vectors ${\overrightarrow{\xi_i \xi_{i, j}} \over \lVert \overrightarrow{\xi_i \xi_{i, j}} \rVert}$ at $\xi_i$ converges to the first vector of the frame $\mathsf{f}_{\xi_{i}}$
as $j \to \infty$.

Let $$\E = \left\{ (\mathsf f_{\xi_1}, \mathsf f_{\xi_2}) \in \tilde  \La : \begin{matrix}
\text{there exists a sequence }
 (\xi_{1, j}, \xi_{2, j}) \in \Lambda \\ \mbox{converging to } (\mathsf f_{\xi_1}, \mathsf f_{\xi_2}) \ (1, 1)\mbox{-tangentially}
 \end{matrix} \right\}.$$
 We first note that $\E$ is non-empty. Since $\Delta$ is Zariski dense in $G$ ,
 $\Delta$ contains a loxodromic element, say, $g \in \Delta$.
 Denote by $y_g\in \F$ the attracting fixed point of $g$. 
 Choose $\zeta \in \La$ which is in general position with $y_{g^{\pm 1}}$. Then  the sequence $g^{\ell} \zeta $ converges to $ y_g$ as $\ell \to +\infty$.  The claim follows from the compactness of the unit sphere in the tangent space of $\F$ at $y_g$.
 
 On the other hand, since
 the action of $G$ on $\F$ is conformal and $\La$ is $\D$-invariant,
 $\E$ is a $\Delta$-invariant subset of $\tilde  \La$. Hence by Theorem \ref{GR},  $$\overline{\E}=\tilde \La.$$

Let $Y = (C, S) \in \Upsilon_{\Delta}$.  We will construct a sequence $Y_k \in \Upsilon_{\Delta}^{*}$ converging to $Y$ as $k \to \infty$. 
Choose $\xi = (\xi_1, \xi_2) \in Y \cap \La$.
Choose a unit vector $\mathsf v_1$ at $\xi_1$ tangent to $C$
and a unit vector $\mathsf v_2$ at $\xi_2$  tangent to $S$. For each $i = 1, 2$, choose an orthonormal frame $\mathsf f_{\xi_i}$ in $\F_i$ based at $\xi_i$ whose first vector is $\mathsf v_i$. Since $(\mathsf f_{\xi_1}, \mathsf f_{\xi_2}) \in \tilde \La$ and
$\cal E$ is dense in $\tilde  \La$, we can find a sequence $(\mathsf f_{\eta_{1, k}}, \mathsf f_{\eta_{2, k}}) \in \cal E$ converging to $(\mathsf f_{\xi_1}, \mathsf f_{\xi_2})$ as $k \to \infty$.
Hence, for each $k$, there exists a sequence $\{(\eta^{(k)}_{1, j}, \eta^{(k)}_{2, j}) \in \La:j=1,2, \cdots\}$ converging $(1, 1)$-tangentially to $(\mathsf f_{\eta_{1, k}}, \mathsf f_{\eta_{2, k}})$ as $j \to \infty$. Since $(\mathsf f_{\eta_{1, k}}, \mathsf f_{\eta_{2, k}}) \to (\mathsf f_{\xi_1}, \mathsf f_{\xi_2})$ as $k \to \infty$, we can choose large enough $j_k$ for each $k$  so that the following holds
for each $i = 1, 2$:
\begin{enumerate}
    \item $ \eta_{i, j_k}^{(k)} \to \xi_i$ as $k \to \infty$; and
    \item the unit tangent vector ${\overrightarrow{\eta_{i, k} \eta_{i, j_k}^{(k)}} \over \lVert \overrightarrow{\eta_{i, k} \eta_{i, j_k}^{(k)}}\rVert}$ at $\eta_{i, k}$ converges to $\mathsf v_i$ as $k \to \infty$.
\end{enumerate}

Now we are ready to construct a sequence $Y_k=(C_k, S_k) \in \Upsilon_{\Delta}^{*}$:
\begin{enumerate}
    \item Fix $z_1 \in C - \{\xi_1 \}$ and let $C_k$ be the circle passing through $z_1$, $\eta_{1, k}$ and $ \eta_{1, j_k}^{(k)}$.
    \item Fix $z_2 \in S - \{ \xi_2 \}$. 
    The tangent space $\mathsf{T}_{\xi_2} S$ of $S$ at $\xi_2$ is a codimension one subspace of the tangent space $\mathsf{T}_{\xi_2}\F_2$. Noting that $\mathsf v_2 \in \mathsf{T}_{\xi_2} S $, we can choose unit tangent vectors $\mathsf w_1, \cdots, \mathsf w_{m-2} \in \mathsf{T}_{\xi_2}S$ so that $\mathsf v_2, \mathsf w_1, \cdots, \mathsf w_{m-2}$ form a basis of $\mathsf{T}_{\xi_2} S$.
For each $\ell = 1, \cdots, m - 2$, we choose a sequence $\zeta_{\ell, k} \in \F_2$ converging to $\xi_2$ such that 
     the unit vectors ${\overrightarrow{\eta_{2, k} \zeta_{\ell, k}} \over \lVert \overrightarrow{\eta_{2, k} \zeta_{\ell, k}} \rVert }$  converges to $\mathsf w_{\ell}$ as $k \to~\infty$.
     Then for each $k\ge 1$ large enough, 
the set $$\{ z_2, \eta_{2, k},
\eta_{2, j_k}^{(k)}, \zeta_{1, k}, \cdots, \zeta_{m - 2, k}\}$$ has cardinality $(m+1)$ and hence  uniquely determines an $(m-1)$-dimensional sphere in $\F_2=\S^m$, which we set to be $S_k$.
\end{enumerate}
Since $(C_k, S_k) \cap \La$ contains two distinct points
$(\eta_{1, k}, \eta_{2, k})$ and $  ( \eta_{1, j_k}^{(k)},  \eta_{2, j_k}^{(k)})$, we have
$$(C_k, S_k) \in \Upsilon_{\Delta}^{*}.$$ Moreover, as $k \to \infty$, $C_k$ converges to the unique circle passing through $z_1$ and tangent to $\mathsf v_1$ which must be $C$, and $S_k$ converges to the unique sphere passing through $z_2$ and whose tangent space at $\xi_2$ is same as $\mathsf{T}_{\xi_2}S$, which must be $S$. Therefore $(C_k, S_k)\in \Upsilon_{\Delta}^{*}$ converges to $Y = (C, S)$.
This finishes the proof of Theorem \ref{two}.

\subsection*{Dense orbits}
Let $\G<\so(n+1,1)$ be a convex cocompact subgroup where $n \ge 2$. Then
$\nu_\Ga$ is equal to $\delta_\Ga$-dimensional Hausdorff measure $\cal H^{\delta_\Ga}|_\La$ and $\delta:=\delta_\Ga$ is equal to
the Hausdorff dimension of $\La$ by \cite{Sullivan1979density}.
Let $\rho:\G\to \so(m+1,1)$ be a Zariski dense convex cocompact faithful representation. Let $\Ga_\rho:=(\id\times \rho)(\Ga)<G$ and 
\be \label{eqn.def of Upsilon}
\Upsilon_{\rho} := \Upsilon_{\Gr}  = \{Y=(C,S)
\in \Upsilon : Y \cap \La_\rho \neq \emptyset\}.
\ee
\begin{theorem} \label{ett}
Suppose that $\Ga_\rho$ is Zariski dense.
Then there exists a  $\cal H^{\delta}|_{\La}$-conull $\La'\subset \La$ such that 
for any $Y\in \Upsilon_{\rho}$ intersecting $(\id\times f)(\La')$ non-trivially,
$$\overline{\Gr Y}=\Upsilon_{\rho}.$$ 
\end{theorem}
\begin{proof}
Since $G$ acts transitively on $\Upsilon$ as homeomorphisms,
we have the homeomorphism
$$\Upsilon \simeq G/H$$
where $H=\Stab(Y_0)$ is the stabilizer of some $Y_0=(C_0, S_0)\in  \Upsilon$. Noting that
$H^\circ $ is a semisimple real algebraic subgroup conjugate to $(\so(2,1)\times \op{SO}(n-1))\times \so(m,1) $,
we may choose $Y_0$ so that $H\supset A$ and that $H\cap P$ is a minimal 
parabolic subgroup of $H$.

Recall the subset $\Upsilon_{\rho}^*=\{Y\in \Upsilon_{\rho}: \# Y\cap \La_\rho \ge 2\}$. Let $\tilde \Omega_{\rho}\subset G$ be the preimage of  $\Omega_{\rho}$ for the projection $G\to \Ga_\rho\ba G$. 
Clearly, $\Upsilon_\rho^*\supset \tilde \Omega_\rho Y_0$. In fact, 
we have $\Upsilon_\rho^*=\tilde \Omega_\rho Y_0$. Indeed,
 as $Y_0$ corresponds to $H$,
denoting by $e=(e_1, e_2)\in H$ the identity element, we have $e_1^{\pm}\in C_0$ and $e_2^{\pm}\in S_0$. For any $Y=(C, S) \in \Upsilon_{\rho}^*$,
there exist distinct  $\xi,\eta \in \La\cap C$ such that $f(\xi), f(\eta)\in S$.
We can then find $g_1\in G_1$ such that
 $g_1(C_0)=C$ and $g_1 e_1^+=\xi$
and $g_1e_1^-=\eta$.  Similarly, we can find $g_2\in G_2$ such that
 $g_2(S_0)=S$ and $g_2 e_2^+=f(\xi)$
and $g_2e_2^-=f(\eta)$.
Then $Y = g Y_0$ for $g=(g_1, g_2) \in G$. Since $g^+ = (\xi, f(\xi))$ and $g^-  = (\eta, f(\eta))$, $g\in \tilde \Omega_\rho$. Therefore, $\tilde \Omega_{\rho} Y_0 = \Upsilon_{\rho}^*$.

Suppose that there exists $g\in G$ such that the closure of $[g]A_u^+$ contains $\Omega_{\rho}$ for some $u\in \inte\fa^+$.
Since $A_u^+\subset H$, the closure of $\Ga_\rho g H$ contains $\tilde \Omega_\rho H$, in other words,
the closure of
$\Ga_\rho gY_0$ contains $\tilde \Omega_\rho Y_0=\Upsilon_{\rho}^*$.  Hence by  Theorem \ref{two},
$$\overline{\Ga_\rho gY_0}=\Upsilon_\rho .$$

Since $\G<\so(n+1,1)$ is convex cocompact, we have that $\cal H^{\delta}|_\La$ is
the unique $\Ga$-conformal measure on $\La$, up to a constant multiple \cite{Sullivan1979density}.
Therefore Theorem \ref{ett} follows from Theorem \ref{ttt} and Lemma \ref{lem.pushforward}.
\end{proof}

\section{Doubly stable condition} \label{sec.wKM}
In this section, let $\Ga < \so(n + 1, 1)$ be a discrete group, $n \ge 2$, which is not necessarily convex cocompact.  Let $\La\subset \S^n$ denote its limit set.

We say that a circle $C\subset \S^n$ is $\La$-{\it doubly stable} if 
for any sequence of circles $C_k$ converging to $C$, 
$$\#\limsup (C_k\cap \La) \ge 2.$$
If $\Omega$ is disconnected, 
there exists a $\La$-doubly stable circle (Lemma \ref{lem.doubly}). Recall from \eqref{eqn.def of Upsilon} that $\Upsilon_{\rho} = \{ Y \in \Upsilon : Y \cap \La_{\rho} \neq \emptyset \}$.

\begin{theorem} \label{thm.wKM}
Let $\Ga < \so(n + 1, 1)$ be a discrete subgroup and $\rho:\Ga\to \so(m+1,1)$, $m \ge 2$,  be a discrete faithful representation
with a boundary map $f:\La\to \S^m$.
Assume that there exists at least one $\La$-doubly stable circle.
If $(C_0, S_0)\in \Upsilon_\rho$
 such that $f(C_0 \cap \La) \subset S_0$, then
$$\overline{\Gr(C_0, S_0)}\ne \Upsilon_\rho.$$
\end{theorem}

\begin{proof} 
Let $C \subset \S^n$ be a $\La$-doubly stable circle.
Then for any sequence of circles $C_k \subset \S^n$ converging to $C$ as $k \to \infty$, we have 
\be \label{eqn.untouchable2}
\# \limsup (C_k \cap \La) \ge 2.
\ee  
It follows that $\# C \cap \La \ge 2$.

We first claim that there exists a codimension one sphere $S \subset \S^m$ such that 
\be \label{eqn.nicesphere}
\# S \cap f(C \cap \La) = 1.
\ee
Since $C\cap \La$ is not homemorphic to $\S^m$, $m \ge 2$, the image $f(C \cap \La)$ is a proper compact subset of $\S^m$. 
Therefore we can find a minimal closed $m$-ball $B \subset \S^m$ containing $f(C \cap \La)$. 
By the minimality of $B$, there exists $\xi_0 \in C \cap \La$ such that $f(\xi_0)$ lies in the boundary of $B$. Now any codimension one sphere $S$ in $\S^m$ such that $S \cap B = \{f(\xi_0)\}$ satisfies \eqref{eqn.nicesphere}.

Set $Y = (C, S)$. Since $(\xi_0, f(\xi_0)) \in (C, S)$, we have $Y \in \Upsilon_{\rho}$. We claim that for any $(C_0, S_0)\in \Upsilon_\rho$
 such that $f(C_0 \cap \La) \subset S_0$,
 we have $Y \not \in \overline{\Gr (C_0, S_0)}$; this implies the theorem.
 Suppose not. Then  there exists a sequence $\ga_k \in \Ga$ such that
  $\ga_k C_0 \to C$ and $\rho(\ga_k) S_0 \to S$ as $k \to \infty$. By \eqref{eqn.untouchable2}, we have
  \be \label{at22} 
  \# \limsup (\ga_k C_0 \cap \La) \ge 2.
  \ee
By the $\rho$-equivariance of $f$, we have
$$  f( \ga_k C_0 \cap \La ) =f(\ga_k (C_0\cap \La))=
\rho(\ga_k) f(C_0\cap \La) \subset \rho(\ga_k) S_0  .$$ 
Hence
$$\limsup f( \ga_k C_0 \cap \La ) \subset \limsup 
  \rho(\ga_k) S_0  =  S.$$
 Since $\limsup f(\ga_k C_0 \cap \La) \subset f(C \cap \La)$ and $f$ is injective,
 it follows from \eqref{at22} that
  $\# S \cap f(C \cap \La) \ge 2.$
  This contradicts \eqref{eqn.nicesphere}, proving the claim.
\end{proof}

We say that $\La$ is {\it doubly stable} if for any $\xi\in \La$,
there exists a $\La$-doubly stable circle containing $\xi$.

\begin{lemma} \label{lem.doubly}
Let $\Ga < \so(n + 1, 1)$ be a discrete subgroup.
If $\Omega$ is disconnected, 
then $\La$ is doubly stable.
\end{lemma} 

\begin{proof}
    Let $\Omega_{1}$, $\Omega_2$ be distinct connected components of $\Omega$ and fix any $\xi \in \La$. Let $C$ be a circle containing $\xi$ and intersecting $\Omega_1$ and $\Omega_2$.

    Let $C_k$ be a sequence of circles converging to $C$ as $k \to \infty$. We claim that $\# \limsup (C_k \cap \La) \ge 2.$
Suppose that  $\# \limsup (C_k \cap \La) \le 1$. We will show that
$C\cap \Omega_1$ is a singleton, which is a contradiction since $C\cap \Omega_1$ is an open subset of $C$.

  For each $k$, let $I_k \subset C_k$ be a compact interval containing $C_k \cap \La$ with minimal diameter. 
  Since $C_k - I_k$ is a connected subset of $\Omega$, $C_k - I_k \subset W_k$ for some connected component $W_k$ of $\Omega$. After passing to a subsequence and relabeling $\Omega_1$ and $\Omega_2$ if necessary, we may assume that $\Omega_1 \neq W_k$ and hence $\Omega_1\cap W_k=\emptyset$ for all $k$. 

Let $x, y\in C\cap \Omega_1$. Since the sequence $C_k$ converges to $C$, 
 $x=\lim_{k\to \infty} x_{k}$ and  $y=\lim_{k\to \infty} y_{k}$ for some $x_{k}, y_k \in C_{k}$.
    Since $\Omega_1$ is open, we may assume that  $x_{k}, y_k \in C_{k}\cap \Omega_1$ for all $k\ge 1$. Hence $x_{k}, y_k \notin W_{k}$; so $x_{k}, y_k \in I_{k}$.

Since $\# \limsup (C_k \cap \La) \le 1$, the diameter of $ I_k $ tends to $ 0$ as $k \to \infty$.
Therefore the distance between $x_k$ and $y_k$ must go to $0$ and hence $x=y$.
This proves the claim, finishing the proof.
\end{proof}

\section{Rigidity via circular slices} \label{sec.proof}
Let $n, m\ge 2$. Let $\G<\so(n+1,1)$ be a  Zariski dense convex cocompact subgroup.
Let $\rho:\Ga\to \so (m+1,1)$
be a Zariski dense convex cocompact deformation and $f:\La\to \S^m$ be its boundary map. 
Recall 
$$
\La_f= \bigcup \left\{ C \cap \La : \begin{matrix}
C\subset\S^n \mbox{ is a circle such that} \\
f(C \cap \La) \mbox{ is contained in a } (m-1)\mbox{-sphere of $\S^m$}
\end{matrix}
\right\} .
$$
Theorem \ref{high} is a special case of the following:

\begin{theorem} \label{high2} Suppose that
there exists a $\La$-doubly stable circle. 
Then
$$\text{either}\quad  \La_f=
\La \quad \text{ or } \quad \cal H^{\delta}(\La_f) =0.$$

In the former case, we have
 $n=m$, $f$ extends to some $g\in \Mob(\S^n)$ and $\rho$ is a conjugation by $g$.
 
\end{theorem}

\begin{remark}
    By Lemma \ref{lem.doubly},
when $\Omega$ has at least two components, there exists a $\La$-doubly stable circle. Hence Theorem \ref{high2} applies to this case.
\end{remark}

\begin{proof} 
Suppose that  $\cal H^{\delta} (\La_f) > 0$. We need to show that $\La_f=\La$.
We claim that $\Gr$ cannot be Zariski dense in $G$. 
Suppose that $\Ga_{\rho}$ is Zariski dense. Let
${\La}'\subset \La$ be the $\cal H^{\delta}|_{\La}$-conull subset
given by Theorem \ref{ett}. Since $\cal H^{\delta}(\La_f)>0$,
there exists $\xi_0 \in \La_f\cap \La'$.
By the definition of $\La_f$, we can find $Y_0=(C_0, S_0)\in \yD$ so that $Y_0\ni (\xi_0, f(\xi_0))$ and $f(C_0\cap \La)\subset S_0$. 
By the definition of $\La'$ as in Theorem \ref{ett},
we have
 $$\overline{\Gr Y_0} = \Upsilon_{\rho}.$$
On the other hand, since there exists a $\La$-doubly stable circle,
Theorem \ref{thm.wKM} implies
that  $\overline{\Gr Y_0} \ne \Upsilon_{\rho}.$
This yields a contradiction, proving that $\Ga_{\rho}$ is not Zariski dense. Hence
 by Theorem \ref{Zdense}, $\rho$ extends to a Lie group isomorphism $\so(n+1,1)\to \so(m+1,1)$ and in particular $n=m$. Since the Lie group automorphism of $\so(n+1,1)$
 is a conjugation by some $g\in \Mob(\S^n)$, it follows that  $\rho$ is a conjugation by $g$ and by the uniqueness of the $\rho$-boundary map,
 $f$ is the restriction of $g$ to $\La$. Therefore $\La_f=\La$.
\end{proof}

\subsection*{Topological version without convex cocompactness}
The assumption that $\Ga$ and $\rho(\Ga)$ are convex cocompact was
used to apply the ergodicity as in Theorem \ref{er}. 
The approach of our paper proves the 
following theorem without the convex cocompact hypothesis, which was shown in \cite{kim2022rigidity} for $n=m=2$:
\begin{theorem} \label{thm.topological}
Let $\Ga < \so(n + 1, 1)$ be a Zariski dense discrete subgroup.
Suppose that there exists a $\La$-doubly stable circle.
Let $\rho : \Ga \to \so(m + 1, 1)$ be a Zariski dense deformation with a $\rho$-boundary map $f : \La \to \S^{m}$. Then
$$\mbox{ either } \quad \La_f = \La \quad \mbox{or} \quad \La_f \mbox{ has empty interior in }\La.$$

In the former case, we have $n = m$, $f$ extends to some $g \in \Mob(\S^n)$ and $\rho$ is a conjugation by $g$.
\end{theorem}
 For this, we need to replace
 the ergodicity theorem (Theorem \ref{er}) by the following theorem of Chow-Sarkar for $\Delta=\Ga_\rho$: 
\begin{theorem}[{\cite[Theorem 8.1]{chow2021local}}] \label{chow} Let $\D<G$ be a Zariski dense discrete subgroup.
For any $u\in \inte \L_\D$, there exists a dense $A_u^+$-orbit in
$$\Omega_\D:=\{[g]\in \D\ba G: g^{\pm}\in \La_\D\}.$$
\end{theorem}
This theorem provides  a dense subset $\La' \subset \La$ such that for any $Y \subset\Upsilon_{\rho}$ intersecting $(\id \times f)(\La')$ non-trivially, $\Gr Y$ is dense in $\Upsilon_{\rho}$, which is a topological version of Theorem \ref{ett}. With this replacement, the rest of the proof can be repeated in verbatim.
Theorem \ref{thm.McMullen} is a direct consequence of Theorem \ref{thm.topological} and Lemma \ref{lem.doubly}.

\subsection*{Added in proofs:} We explain the following extension of Theorem \ref{high2}:
\begin{theorem} \label{high5} Let $\Ga < \so(n+1, 1)$ be a Zariski dense discrete subgroup of divergence type and $\rho:\Ga\to \so (m+1,1)$ is a Zariski dense quasi-isometric deformation (i.e., one inducing a quasi-isometric embedding of $\Ga o \subset \H^{n+1}$ into $\bH^{m+1}$) with boundary map $f:\La\to \S^m$. Let $\nu_\Ga$ be the 
unique $\delta$-dimensional $\Ga$-conformal measure on $\La$ where $\delta$ denotes the critical exponent of $\Ga$.  Suppose that
there exists a $\La$-doubly stable circle. 
Then
$$\text{either}\quad  \La_f=
\La \quad \text{ or } \quad \cal \nu_\Ga(\La_f) =0.$$

In the former case, we have
 $n=m$, $f$ extends to some $g\in \Mob(\S^n)$ and $\rho$ is a conjugation by $g$.
 \end{theorem}

As before, let $\Ga_{\rho}: = (\id \times \rho)(\Ga) < G := \so(n+1, 1) \times \so(m+1, 1)$. Let  $m_{\rho}^{\BMS}$ be the Bowen-Margulis-Sullivan measure on $\Gr \ba G$ associated to $\nu_{\rho}:= (\id \times f)_* \nu_{\Ga}$. Its support  is 
$
\Omega_{\rho} := \{[g] \in \Ga_\rho \ba G : g^{\pm} \in \La_{\rho} \}.
$ Since  $\Ga_\rho$ is a transverse subgroup of $G$ and hence hyper-transverse in the sense of \cite{kim2024conformal}, it follows from \cite[Theorem 1.14]{kim2024conformal} that  $m_{\rho}^{\BMS}$-a.e. $A^+$-orbits are dense in $\Omega_\rho$.
Since the limit cone $\L_{\rho} $ is contained in
$ \inte \fa^+ \cup \{0\} $, there exists a closed convex cone $\cal C \subset \inte \fa^+ \cup \{0\}  $ whose interior contains $\L_\rho - \{0\}$. Then a.e. $\exp \cal C$-orbits are dense in $\Omega_\rho$ (cf. \cite[Lemma 7.2]{lee2020invariant}).
Observing that the conjugation action of $\exp \cal C$ on $N$ has a uniform contraction property,
we can repeat the proof of Theorem \ref{high2}, replacing the directional flow by $\exp \C$-flow, and $\cal H^\delta$ by $\nu_\Ga$.

\bibliographystyle{plain}

\begin{thebibliography}{10}

\bibitem{Benoist1997proprietes}
Y.~Benoist.
\newblock Propri\'{e}t\'{e}s asymptotiques des groupes lin\'{e}aires.
\newblock {\em Geom. Funct. Anal.}, 7(1):1--47, 1997.

\bibitem{Bridson1999metric}
M.~Bridson and A.~Haefliger.
\newblock {\em Metric spaces of non-positive curvature}, volume 319 of {\em
  Grundlehren der mathematischen Wissenschaften [Fundamental Principles of
  Mathematical Sciences]}.
\newblock Springer-Verlag, Berlin, 1999.

\bibitem{Burger1993intersection}
M.~Burger.
\newblock Intersection, the {M}anhattan curve, and {P}atterson-{S}ullivan
  theory in rank {$2$}.
\newblock {\em Internat. Math. Res. Notices}, (7):217--225, 1993.

\bibitem{burger2021hopf}
M.~Burger, O.~Landesberg, M.~Lee, and H.~Oh.
\newblock The {H}opf--{T}suji--{S}ullivan dichotomy in higher rank and
  applications to {A}nosov subgroups.
\newblock {\em J. Mod. Dyn.}, 19:301--330, 2023.

\bibitem{CZZ_transverse}
R.~Canary, T.~Zhang, and A.~Zimmer.
\newblock Patterson-{S}ullivan measures for transverse subgroups.
\newblock {\em J. Mod. Dyn.}, 20:319--377, 2024.

\bibitem{chow2021local}
M.~Chow and P.~Sarkar.
\newblock Local mixing of one-parameter diagonal flows on {A}nosov homogeneous
  spaces.
\newblock {\em Int. Math. Res. Not. IMRN}, (18):15834--15895, 2023.

\bibitem{DalBoKim_criterion}
F.~Dal'Bo and I.~Kim.
\newblock A criterion of conjugacy for {Z}ariski dense subgroups.
\newblock {\em C. R. Acad. Sci. Paris S\'er. I Math.}, 330(8):647--650, 2000.

\bibitem{edwards2020anosov}
S.~Edwards, M.~Lee, and H.~Oh.
\newblock Anosov groups: local mixing, counting and equidistribution.
\newblock {\em Geom. Topol.}, 27(2):513--573, 2023.

\bibitem{Gromov1981hyperbolic}
M.~Gromov.
\newblock Hyperbolic manifolds.
\newblock In {\em Bourbaki {S}eminar, {V}ol. 1979/80}, volume 842 of {\em
  Lecture Notes in Math.}, pages 40--53. Springer, 1981.

\bibitem{Guichard2012anosov}
O.~Guichard and A.~Wienhard.
\newblock Anosov representations: domains of discontinuity and applications.
\newblock {\em Invent. Math.}, 190(2):357--438, 2012.

\bibitem{Guivarch2007actions}
Y.~Guivarc'h and A.~Raugi.
\newblock Actions of large semigroups and random walks on isometric extensions
  of boundaries.
\newblock {\em Ann. Sci. \'{E}cole Norm. Sup. (4)}, 40(2):209--249, 2007.

\bibitem{Kapovich2017anosov}
M.~Kapovich, B.~Leeb, and J.~Porti.
\newblock Anosov subgroups: dynamical and geometric characterizations.
\newblock {\em Eur. J. Math.}, 3(4):808--898, 2017.

\bibitem{kim2024conformal}
D.~M. Kim.
\newblock Conformal measure rigidity and ergodicity of horospherical
  foliations.
\newblock {\em arXiv preprint arXiv:2404.13727}, 2024.

\bibitem{kim2022rigidity}
D.~M. Kim and H.~Oh.
\newblock Rigidity of {K}leinian groups via self-joinings.
\newblock {\em Invent. Math.}, 234(3):937--948, 2023.

\bibitem{kim2023rigidity}
D.~M. Kim and H.~Oh.
\newblock Conformal measure rigidity for representations via self-joinings.
\newblock {\em Adv. Math.}, 458:Paper No. 109992, 40, 2024.

\bibitem{KOW_indicators}
D.~M. Kim, H.~Oh, and Y.~Wang.
\newblock Properly discontinous actions, growth indicators and conformal
  measures for transverse subgroups.
\newblock {\em Preprint arXiv:2306.06846}, 2023.

\bibitem{Labourie2006anosov}
F.~Labourie.
\newblock Anosov flows, surface groups and curves in projective space.
\newblock {\em Invent. Math.}, 165(1):51--114, 2006.

\bibitem{lee2020invariant}
M.~Lee and H.~Oh.
\newblock Invariant measures for horospherical actions and {A}nosov groups.
\newblock {\em Int. Math. Res. Not. IMRN}, (19):16226--16295, 2023.

\bibitem{Maskit1988Kleinian}
B.~Maskit.
\newblock {\em Kleinian groups}, volume 287 of {\em Grundlehren der
  mathematischen Wissenschaften}.
\newblock Springer-Verlag, Berlin, 1988.

\bibitem{Matsuzaki1998hyperbolic}
K.~Matsuzaki and M.~Taniguchi.
\newblock {\em Hyperbolic manifolds and {K}leinian groups}.
\newblock Oxford Mathematical Monographs. Oxford University Press, 1998.

\bibitem{McMullen1994classification}
C.~McMullen.
\newblock The classification of conformal dynamical systems.
\newblock In {\em Current developments in mathematics, 1995, {\rm 323--360}}.
  Int. Press, Cambridge, MA, 1994.

\bibitem{Mostowbook}
G.~Mostow.
\newblock {\em Strong rigidity of locally symmetric spaces}.
\newblock Annals of Mathematics Studies, No. 78. Princeton University Press,
  Princeton, N.J.; University of Tokyo Press, Tokyo, 1973.

\bibitem{Sullivan1979density}
D.~Sullivan.
\newblock The density at infinity of a discrete group of hyperbolic motions.
\newblock {\em Inst. Hautes \'{E}tudes Sci. Publ. Math.}, (50):171--202, 1979.

\bibitem{Sullivan1985quasiconformal}
D.~Sullivan.
\newblock Quasiconformal homeomorphisms and dynamics. {I}. {S}olution of the
  {F}atou-{J}ulia problem on wandering domains.
\newblock {\em Ann. of Math. (2)}, 122(3):401--418, 1985.

\bibitem{thurston2022geometry}
W.~Thurston.
\newblock {\em The Geometry and Topology of Three-Manifolds}, volume~27.
\newblock American Mathematical Society, 2022.

\bibitem{Tukia1985isomorphisms}
P.~Tukia.
\newblock On isomorphisms of geometrically finite {M}\"{o}bius groups.
\newblock {\em Inst. Hautes \'{E}tudes Sci. Publ. Math.}, (61):171--214, 1985.

\bibitem{Zhang2022construction}
Y.~Zhang.
\newblock Construction of acylindrical hyperbolic 3-manifolds with
  quasifuchsian boundary.
\newblock {\em Exp. Math.}, 31(3):883--896, 2022.

\end{thebibliography}

\end{document}